\documentclass[11pt,reqno]{amsart}
\usepackage{amsmath, amsthm, amsfonts, amssymb,color,url,enumitem}

\usepackage{tikz,amssymb,float}
\usetikzlibrary{arrows.meta,calc}
\usepackage{graphicx,paralist}
\makeatletter
\newcommand*\bigcdot{\mathpalette\bigcdot@{.5}}
\newcommand*\bigcdot@[2]{\mathbin{\vcenter{\hbox{\scalebox{#2}{$\m@th#1\bullet$}}}}}
\makeatother

\usepackage{cancel}
\usepackage{wrapfig,hyperref}
\usepackage{mathtools}
\usepackage{fullpage}
\usepackage{subfig}
\usepackage{color}
\usepackage[normalem]{ulem}
\usepackage{tikz}
\usepackage{hyperref}
\hypersetup{
    colorlinks=true,
    linkcolor=blue,
    filecolor=magenta,
    urlcolor=cyan,
}
\theoremstyle{definition}
\newtheorem{definition}{Definition}[section]
\newtheorem{theorem}{Theorem}
\newtheorem{lemma}{Lemma}
\newtheorem{cor}{Corollary}

\newtheorem{remark}{Remark}

\newtheorem{conjecture}{Conjecture}

\newcommand{\Ac}{\mathcal{A}}
\newcommand{\Bc}{\mathcal{B}}

\newcommand{\Ec}{\mathcal{E}}

\newcommand{\Ic}{\mathcal{I}}

\newcommand{\Pc}{\mathcal{P}}

\newcommand{\Rc}{\mathcal{R}}

\newcommand{\Xc}{\mathcal{X}}

\newtheorem{claim}{Claim}
\newtheorem{claimproof}{Proof of Claim}

\usepackage{amssymb,latexsym}
\usepackage{graphicx}
\usepackage{enumerate}
\usepackage{color,amsmath,amssymb,amsthm,fullpage}
\usepackage{url}
\usepackage{comment}
\usepackage{subfiles}

\author[J.~De Silva]{Jessica ~De Silva}
\address{
Department of Mathematics \\
California State University, Stanislaus\\
1 University Cir, Turlock, CA 95382\\
} 
\email{jdesilva1@csustan.edu}

\author[A.~B.~Dionne]{Adam B.~Dionne}
\address{
Department of Mathematics and Statistics \\
Williams College \\
2679 Paresky \\
Williamstown, MA 01267, USA} 
\email{abd2@williams.edu}

\author[A.~Dunkelberg]{Aidan ~Dunkelberg}
\address{
Department of Mathematics and Statistics \\
Williams College \\
33 Stetson Court \\
Williamstown, MA 01267, USA} 
\email{awd4@williams.edu}

\author[P.~E.~Harris]{Pamela E.~Harris}
\address{
Department of Mathematics and Statistics \\
Williams College \\
33 Stetson Court \\
Williamstown, MA 01267, USA} 
\email{peh2@williams.edu}
\thanks{P.~E.~Harris was supported by a Karen Uhlenbeck EDGE Fellowship.}

\title{Very well-covered graphs with the Erd\H{o}s-Ko-Rado property}

\begin{document}
\maketitle

\begin{abstract}
    A family of independent $r$-sets of a graph $G$ is an $r$-star if every set in the family contains some fixed vertex $v$. A graph is $r$-EKR if the maximum size of an intersecting family of independent $r$-sets is the size of an $r$-star. Holroyd and Talbot conjecture that a graph is $r$-EKR as long as $1\leq r\leq\frac{\mu(G)}{2}$, where $\mu(G)$ is the minimum size of a maximal independent set. It is suspected that the smallest counterexample to this conjecture is a well-covered graph. Here we consider the class of very well-covered graphs $G^*$ obtained by appending a single pendant edge to each vertex of $G$. We prove that the pendant complete graph $K_n^*$ is $r$-EKR when $n \geq 2r$ and strictly so when $n>2r$. Pendant path graphs $P_n^*$ are also explored and the vertex whose $r$-star is of maximum size is determined.
\end{abstract}

\section{Introduction}
Let $G$ be a finite simple graph with vertex and edge sets $V(G)$ and $E(G)$, respectively. Let $\Ic^{(r)}(G)$ denote the family of all independent $r$-sets of $G$. We say that $\Ac\subseteq\Ic^{(r)}(G)$ is {\bf intersecting} if the intersection of each pair of sets in $\Ac$ is nonempty. One such intersecting family is an {\bf $r$-star} comprised of independent $r$-sets containing some fixed vertex $v\in V(G)$. The vertex $v$ is called the {\bf centre} of the $r$-star. A graph is {\bf $r$-EKR} if there exists an $r$-star whose size is maximum amongst all intersecting families of independent $r$-sets. If all extremal intersecting families are $r$-stars, the graph is considered to be strictly $r$-EKR. This naming stems from the classical Erd\H{o}s-Ko-Rado theorem, framed in the language of graph theory as follows:
\begin{theorem}[\emph{Erd\H{o}s-Ko-Rado} \cite{ClassicEKR}]
If $E_n$ is the empty graph of order $n$, then $E_n$ is $r$-EKR for $n \geq 2r$ and strictly so when $n>2r$. 
\end{theorem}

This Erd\H{o}s-Ko-Rado property for graphs was formalized by Holroyd, Talbot, and Spencer in \cite{EKRComp}. There has been significant interest and progress in exploring the Erd\H{o}s-Ko-Rado property for many classes of graphs (e.g., \cite{DisUnions,SpecTrees,Chordal,LadderGraph}). In \cite{EKRProp}, Holroyd and Talbot conjecture a connection between the Erd\H{o}s-Ko-Rado property of a graph $G$ and the parameter $\mu(G)$ denoting the minimum size of a maximal independent set.
\begin{conjecture}[\emph{Holroyd and Talbot} \cite{EKRProp}]
\label{HT}
A graph $G$ will be $r$-EKR if $1 \leq r \leq \frac{\mu(G)}{2}$, and is strictly so if $2<r<\frac{\mu(G)}{2}.$
\end{conjecture}

In \cite{borg1,borg2}, Borg proved this conjecture in a much more general form for any graph $G$ with $\mu(G)\geq(r-1)r^2+r$. This result is a major step forward in proving Conjecture \ref{HT}, however there is still a need to tighten the bound on $\mu(G)$. While introducing this conjecture, Holroyd and Talbot suggest that if a counterexample exists, it is likely that $\mu(G)$ is as large as possible. This occurs exactly when a graph is {\bf well-covered}, that is every maximal independent set is also of maximum size. This theory motivates the study of the Erd\H{o}s-Ko-Rado property of well-covered graphs.

In \cite{well_covered}, Finbow, Hartnell, and Nowakowski characterize well-covered graphs with girth at least 6. A vertex $p\in V(G)$ whose neighborhood contains a single vertex $x$ is considered to be a {\bf pendant vertex} and $px\in E(G)$ is a {\bf pendant edge}.

\begin{theorem}[\emph{Finbow, Hartnell, and Nowakowski} \cite{well_covered}] Let $G$ be a connected graph of girth at least~$6$, which is neither isomorphic to $C_7$ nor $K_1$. Then $G$ is well-covered if and only if its pendant edges form a perfect matching.
\end{theorem}

 A graph whose pendant edges form a perfect matching can be constructed by appending a single pendant edge to each vertex of some base graph $G$. We denote this well-covered graph as $G^*$ and refer to it as a pendant graph. In addition to being well-covered, a pendant graph is also {\bf very well-covered} as there are no isolated vertices and $\mu(G)=\frac{|V(G)|}{2}$.

Here we aim to motivate the study of pendant graphs in relation to the Erd\H{o}s-Ko-Rado property. In \cite{dis_cliques}, Boll\'{o}bas and Leader prove the first and only prior EKR-type result regarding pendant graphs. They show that the disjoint union of at least $r$ cliques $K_t$ with $t\geq 2$ is $r$-EKR. The special case of size 2 cliques gives that when $n\geq r$, $E_n^*$ is $r$-EKR and strictly so unless $n=r$. A recent result by Estrugo and Pastine proves that every pendant graph has an $r$-star of maximum size whose centre is a pendant. The result requires the existence of an {\bf escape path}, defined to be a path $v_1v_2\cdots v_n$ such that $\textrm{deg}(v_n)=1$ and $\textrm{deg}(v_i)=2$ for $2\leq i\leq n-2$.
\begin{theorem}[\emph{Estrugo and Pastine} \cite{escape_path}]\label{escape} Let $G$ be a graph, $v\in V(G)$, and $r\geq 1$. If there is an escape path from $v$ to a pendant vertex $p$, then $|\Ic^{(r)}_{v}(G)|\leq|\Ic^{(r)}_p(G)|$.
\end{theorem}
In the pendant graph $G^*$, the edge between a vertex $x$ in the base graph $G$ and its corresponding pendant vertex $p_x$ is an escape path from $x$ to $p_x$. This implies that for every vertex $v\in V(G^*)$, $|\Ic^{(r)}_v(G^*)|\leq|\Ic^{(r)}_p(G^*)|$ for some pendant vertex $p$.

In this paper, we prove the following results for the pendant graphs $K_n^*$ and $P_n^*$.
\begin{theorem}\label{pkEKR}
The pendant complete graph $K_n^*$ is $r$-EKR for $n \geq 2r$, and strictly so for $n > 2r$.
\end{theorem}
\begin{theorem}
\label{thetheorem}
Let $1\leq r\leq n$. The $r$-stars whose centre is one of the two pendant vertices appended to the second outermost vertices on the base graph $P_n$ are maximum size $r$-stars of $P_n^*$.
\end{theorem}

Theorem \ref{pkEKR} aligns with Conjecture \ref{HT} and we believe a similar EKR-type result holds for $P_n^*$ although we establish $P_n^*$ is not $n$-EKR. The proof of Theorem \ref{thetheorem} is of particular interest as we establish a recurrence relation that reduces to the Fibonacci sequence when $r=n$.

\section{Pendant Complete Graphs}
We begin by establishing notation used throughout. Given a graph $G$ with vertices $\{x_1, \ldots, x_n\}$, the pendant graph of $G$, denoted $G^*$, is defined by the following: 
\[ V(G^*) = \{x_1, \ldots, x_n\} \sqcup \{p_1, \ldots, p_n\},\] \[E(G^*) = E(G) \sqcup \{x_1p_1, \ldots, x_np_n\}. \]
When referring to the base graph $G$ in regards to $G^*$, it is more precisely defined as $G:=G^*[\{x_1,\ldots,x_n\}]$.

Recall that $\Ic^{(r)}(G^*)$ is the set of all independent $r$-sets of $G^*$. Further, for $v\in V(G^*)$ define $\Ic^{(r)}_v(G^*)$ to be the $r$-star comprised of all independent $r$-sets of $G^*$ containing $v$.
\begin{lemma}
For $n,r\geq 1$,
    \begin{align}
       \left| \Ic^{(r)}(K^*_n) \right| =
       \begin{cases}(r+1) \binom{n}{r} & \text{if }r\leq n\\
       0&\text{otherwise}.
       \end{cases}
    \end{align}
\end{lemma}
\begin{proof}
Suppose $\Ic^{(r)}(K_n^*)$ is nonempty. Note that an independent set of $K_n^*$ can contain at most one vertex of $K_n$. Since there are $n$ pendant vertices, this implies $r\leq n+1$. Sets of size $n+1$ containing at most one vertex of $K_n$ contain a pendant-base vertex pair. Hence no set of size $n+1$ is independent, so we may assume $r\leq n$.

Partition the set of independent sets of size $r\leq n$ into those containing only pendant vertices, and those containing exactly one vertex from $K_n$. The set of $n$ pendant vertices form an independent set, thus there are $\binom{n}{r}$ independent $r$-sets containing only pendant vertices. Each independent set containing one vertex in $K_n$ can be realized as a set of only pendant vertices with one of the $r$ pendants $p_i$ mapped to its corresponding base vertex $x_i$. This gives $r\binom{n}{r}$ independent sets containing at least one vertex from $K_n$, proving the claim.
\end{proof}

The following lemma uses compression-like operation similar to that of \cite{EKRComp}. For $v\in G^*$, we define $G^*\downarrow v:=G^*\backslash N[v]$.

\begin{lemma} \label{cliqstars} Let $p$ be a pendant vertex of $K_n^*$ and suppose $r\leq n$, then
    \begin{align}
        \left| \Ic^{(r)}_p(K_n^*) \right| = \left| \Ic^{(r-1)}(K^*_{n-1}) \right| = r \binom{n-1}{r-1}.
    \end{align}
\end{lemma}
\begin{proof}
Consider the bijection 
\[f:\Ic_p^{(r)}(K_n^*)\to\Ic^{(r-1)}(K_{n}^*\downarrow p)\]
defined as $f(I)=I\backslash\{p\}$. Furthermore, $K_n^*\downarrow p\cong K_{n-1}^*$ as $p$ is a pendant vertex.
\end{proof}

Now, given any independent intersecting family $\Ac$ on $K^*_n$, we look to construct an injective mapping to an independent intersecting family that permits favorable partitions. 

\begin{lemma}\label{pendint}
    Let $\Ac\subseteq\Ic^{(r)}(K^*_n)$ be an intersecting family. Then, there exists an intersecting family $\Bc\subseteq\Ic^{(r)}(K_n^*)$ satisfying
    \vspace{0.2cm}
\begin{compactenum}[\hspace{0.25cm} 1. ]
        \item $|\Ac|=|\Bc|$, and
        \item if $B_1,B_2\in \Bc$, then $B_1\cap B_2\not\subseteq V(K_n)$.
\end{compactenum}
\vspace{0.2cm}
\end{lemma}
\begin{proof}
Let $\Ac_n:=\Ac$ and for each $1\leq i\leq n$ recursively define $\Ac_{i-1}:=\varphi_i(\Ac_i)$ where
\begin{align*}
    \varphi_i(A) = \begin{cases}
    A \cup \{p_i\} \setminus \{x_i\} &\text{ if } \exists \text{ some } C \in \Ac_i \text{ such that } A \cap C = \{x_i\}, \\
    A &\text{ otherwise.}
    \end{cases}
\end{align*}

\begin{claim}
$\Ac_i\subseteq\Ic^{(r)}(K_n^*)$ for all $0\leq i \leq n.$
\end{claim}

\begin{claimproof}
\normalfont
It suffices to show that if $A\in \Ac_i$ is independent, then $\varphi_i(A)\in\Ac_{i-1}$ has size $r$ and is also independent. This is clear if $\varphi_i(A)=A.$

Otherwise $\varphi_i(A)\backslash A=\{p_i\}$. Since $N(p_i)=\{x_i\}$ and $x_i\not\in \varphi_i(A)$ then $\varphi_i(A)$ is independent. By definition of $\varphi_i$, $\varphi_i(A)\neq A$ implies $x_i\in A$. Furthermore $\{p_{i},x_{i}\}\nsubseteq A$ as $A$ is independent, thus $\varphi_{i}(A)$ has size $r$ and $\Ac_{i-1}\subseteq \Ic^{(r)}(K_n^*)$.
\end{claimproof}

\begin{claim}
$\Ac_i$ is an intersecting family for all $0\leq i\leq n$.
\end{claim}

\begin{claimproof}
\normalfont
As in the previous claim, we show that if $\Ac_i$ is an intersecting family, then $\Ac_{i-1}=\varphi_i(\Ac_i)$ is intersecting. Consider $A,C\in \Ac_i$, if $\varphi_i(A)=A$ and $\varphi_i(C)=C$ then $\varphi_i(A)\cap\varphi_i(C)\neq\emptyset$ since $\Ac_i$ is intersecting. Thus we may further suppose $\varphi_i(A)=A\cup\{p_i\}\backslash\{x_i\}$. The intersecting property of $\Ac_i$ implies $A\cap C=\{x_i\}$ or there is some $v\in A\cap C$ where $v\neq x_i$. In the former case, $\varphi_i(C)=C\cup\{p_i\}\backslash\{x_i\}$ hence $\varphi_i(A)\cap\varphi_i(C)=\{p_i\}$. Otherwise, since $v\neq x_i$ $v\in \varphi_i(A)\cap \varphi_i(C)$.
\end{claimproof}

\begin{claim}
$\varphi_i$ is injective for all $0\leq i\leq n$.
\end{claim}

\begin{claimproof}
\normalfont
Consider unique sets $A,C \in \Ac_i$ and suppose for the sake of contradiction $\varphi_i(A)=\varphi_i(C)$. Without loss of generality, assume $\varphi_i(A)=A\cup\{p_i\}\backslash\{x_i\}$ and $\varphi_i(C)=C$. Since $\varphi_i(A)\neq A$, there exists some $D\in \Ac_i$ such that $A\cap D=\{x_i\}$. Furthermore, $\varphi_i(A)=\varphi_i(C)=C$ implies $C=A\cup\{p_i\}\backslash\{x_i\}$. However $\Ac_i$ is intersecting, so $D\cap C=\{p_i\}$ giving $\{x_i,p_i\}\subseteq D$. This contradicts the independence of $D$, therefore $\varphi_i$ is injective.
\end{claimproof}

Consider $\Bc:=\Ac_0$, Claims 1-2 give that $\Bc$ is an intersecting subfamily of $\Ic^{(r)}(K_n^*)$. Furthermore, notice that $\Bc=\varphi(\Ac)$ where
\[\varphi=\varphi_1 \circ \varphi_2 \circ \dots \circ \varphi_n.\]
Since each $\varphi_i$ is injective by Claim 3, $\varphi$ is injective. Thus $|\Ac|=|\Bc|$.

Lastly, consider $B_1,B_2\in \Bc$. We  show $B_1\cap B_2\neq\{x_i\}$ for any $1\leq i\leq n$. If $x_i\in B_1\cap B_2$ for some $i$, then $B_1\backslash\{p_j\}\cup\{x_j\}$ and $B_2\backslash\{p_j\}\cup\{x_j\}$ are not independent for any $j\neq i$. Thus $(\varphi_1\circ\cdots\circ\varphi_{i-1})^{-1}(B_1)=B_1$ and $(\varphi_1\circ\cdots\circ\varphi_{i-1})^{-1}(B_2)=B_2$, so we may write $B_1=\varphi_i(C_1)$ and $B_2=\varphi_i(C_2)$ for some $C_1,C_2\in\Ac_i$. Since $x_i\in B_1\cap B_2$ then $\varphi_i(C_1)=C_1$ and $\varphi_i(C_2)=C_2$ giving $C_1=B_1$ and $C_2=B_2$. Hence $x_i\in B_1\cap B_2$ implies $B_1\cap B_2\neq\{x_i\}$ else $\varphi_i(B_1)\neq B_1$ and $\varphi_i(B_2)\neq B_2$. Therefore if $x_i\in B_1\cap B_2$, then $B_1\cap B_2\neq\{x_i\}$, concluding the proof. 
\end{proof}

\begin{proof}[Proof of Theorem \ref{pkEKR}]
Let $\Ac\subseteq\Ic^{(r)}(K_n^*)$ be an intersecting family and let $\Bc$ be the corresponding intersecting family satisfying the conditions of applying Lemma \ref{pendint} to $\Ac$. Consider the partition $\Bc = \Pc \sqcup \Xc$ where
\[ \Pc := \{B \in \Bc: x_i \notin B \text{ for all } i\}\]
\[ \Xc := \{B \in \Bc: x_i \in B \text{ for some } i \}\subseteq\Ac.\]
Note that $\Pc\subseteq\Ic^{(r)}(K_n^*[\{p_1,\ldots,p_n\}])$ is an intersecting family. Since $K_n^*[\{p_1,\ldots,p_n\}]\cong E_n$, the Erd\H{o}s-Ko-Rado theorem gives $|\Pc|\leq\binom{n-1}{r-1}$ when $n\geq 2r$.

 Recall that an independent set of $K_n^*$ contains at most one vertex of $K_n$, hence each set in the following family of subsets of $\Xc$ are of size $r-1$ and contain only pendants
\[ \Rc := \{ X \setminus \{x_i\} : X \in \Xc,\,x_i\in X \}. \]
Note that each set of $\Xc$ contains $x_i$ for some $i$, thus
\[\Xc\subseteq\{R\cup\{x_j\}\,:\,R\in\Rc, p_j\not\in R\}.\]
Since each $R\in\Rc$ has $n-(r-1)$ choices of $j$ satisfying $p_j\not\in R$ then
\[|\Xc|\leq (n-r+1)|\Rc|.\]

By condition (2) of $\Bc$ in Lemma \ref{pendint}, $\Rc\subseteq\Ic^{(r-1)}(K_n^*[\{p_1,\ldots,p_n\}])$ is intersecting. By the Erd\H{o}s-Ko-Rado theorem, $\Rc$ has size at most $\binom{n-1}{r-2}$ when $n \geq 2(r-1)$. This inequality is strict when $n>2(r-1)$.

It follows that when $n \geq 2r$,
\begin{align*}|\Ac|  &= |\Bc|\\
&= |\Pc| + |\Xc|\\
&\leq |\Pc| + (n-r+1)|\Rc|\\
&\leq \binom{n-1}{r-1} + (n-r+1)\binom{n-1}{r-2}\\
&= r\binom{n-1}{r-1}.
\end{align*}

Therefore $|\Ac|\leq r\binom{n-1}{r-1}$ and Lemma \ref{cliqstars} implies that $K_n^*$ is $r$-EKR when $n\geq 2r$.

Furthermore, the Erd\H{o}s-Ko-Rado Theorem implies this inequality is strict when $n>2r$ unless all of the following statements hold:

\vspace{0.2cm}
\begin{compactenum}[\hspace{0.25cm} 1. ]
    \item $\Pc$ is the set of all $r$-sets of $\{p_1,\ldots,p_n\}$ containing some fixed vertex $p_i$.
    \item $\Xc=\{R\cup\{x_k\}\,:\,R\in\Rc,p_k\not\in R\}$.
    \item $\Rc$ is the set of all $(r-1)$-sets of $\{p_1,\ldots,p_n\}$ containing some fixed vertex $p_j$.
\end{compactenum}
\vspace{0.2cm}
To prove $K_n^*$ is strictly $r$-EKR when $n>2r$, it suffices to show that these conditions imply $p_i=p_j$. Suppose $p_i\neq p_j$. Since $n>2r,$ there exist $2r-2$ distinct pendant vertices
\[\{p_{i_1},\ldots,p_{i_{r-1}},p_{j_1},\ldots,p_{j_{r-1}}\}\subseteq\{p_1,\ldots,p_n\}\backslash\{p_i,p_j\}.\]
Consider the following sets
\[P=\{p_i,p_{i_1},\ldots,p_{i_{r-1}}\}\textrm{ and } X=\{p_j,p_{j_1},\ldots,p_{j_{r-2}},x_{j_{r-1}}\}.\]
From (1), we have $P\in\Pc$ and by (2) and (3), $X\in\Xc\subseteq\Ac$. Since $\Ac$ is intersecting, $X\in\Ac$, and $P\cap X=\emptyset$, then $P\not\in \Ac$. From Lemma 3, this implies either $P\cup\{x_i\}\backslash\{p_i\}\in\Ac$ or $P\cup\{x_{i_k}\}\backslash\{p_{i_k}\}\in\Ac$ for some $1\leq k\leq r-1$. But $x_{j_{r-1}}\not\in\{x_i,x_{i_1},\ldots,x_{i_{r-1}}\}$, thus both cases contradict the intersecting property of $\Ac$. Therefore $p_i=p_j$, hence $K_n^*$ is strictly $r$-EKR when $n>2r$.
\end{proof}

\section{Pendant Path Graphs}\label{sec:results}

Define $P_n$ to be the path graph on $n$ vertices whose edges are of the form $x_ix_{i+1}$ with $1\leq i\leq n-1$. As is standard in many EKR-type results, we establish a recurrence relation for $r$-stars of $P_n^*$. However, the recurrence only holds for $r$-stars whose centre is a pendant vertex.

\begin{lemma}
\label{bigrecur}
The recurrence relation
\begin{align}|\Ic_{p_i}^{(r)}(P_n^*)| = |\Ic_{p_i}^{(r)}(P_{n-1}^*)| + |\Ic_{p_i}^{(r-1)}(P_{n-1}^*)| + |\Ic_{p_i}^{(r-1)}(P_{n-2}^*)| + |\Ic_{p_i}^{(r-2)}(P_{n-2}^*)|\end{align}\label{bigrecrelation}
holds for all $n,r\geq 1$ and $i \leq n-2$.
\end{lemma}

\begin{proof}[Proof of Lemma 1]
Let
\begin{align*}
    \Pc &:= \{ S\backslash\{p_n\} : S\in \Ic_{p_i}^{(r)}(P_n^*),\,p_n \in S \}, \\
    \Xc &:= \{ S\backslash\{x_n\} : S\in \Ic_{p_i}^{(r)}(P_n^*), \,x_n \in S, p_{n-1}\not\in S\},\\
    \Bc &:= \{ S\backslash\{x_n,p_{n-1}\} : S\in \Ic_{p_i}^{(r)}(P_n^*), \,x_n,p_{n-1} \in S\},
    \text{and} \\
    \Ec &:= \{ S \in \Ic_{p_i}^{(r)}(P_n^*) : x_n, p_n \notin S\}.
\end{align*}
Note that $i\leq n-2$ implies $p_i\in V(P_{n-1}^*)$ and $p_i\in V(P_{n-2}^*)$. Moreover, an independent set $S$ of $P_n^*$ satisfies $|S\cap\{x_j,p_j\}|\leq 1$ for all $j$. This gives $\Pc=\Ic^{(r-1)}_{p_i}(P_{n-1}^*)$, $\Bc=\Ic^{(r-2)}_{p_i}(P_{n-2}^*)$, and $\Ec=\Ic_{p_i}^{(r)}(P_{n-1}^*)$. Since $x_n\in S$ implies $x_{n-1}\not\in S$, we have $\Xc=\Ic_{p_i}^{(r-1)}(P_{n-2}^*)$. Therefore
\begin{align*}
|\Ic_{p_i}^{(r)}(P_n^*)| &= |\Ec| + |\Pc| + |\Xc| + |\Bc| \\&=|\Ic_{p_i}^{(r)}(P_{n-1}^*)| + |\Ic_{p_i}^{(r-1)}(P_{n-1}^*)| + |\Ic_{p_i}^{(r-1)}(P_{n-2}^*)| + |\Ic_{p_i}^{(r-2)}(P_{n-2}^*)|,
\end{align*}
as claimed.
\end{proof}

We can extend this recurrence relation to $p_{n-1}$ and $p_n$ by graph symmetry. 

\begin{remark}
\label{symmetry}
When $n \geq 4$ and $1 \leq i, r \leq n$, the following recurrences hold by symmetry:
\[ |\Ic_{p_{n}}^{(r)}(P_n^*)| = |\Ic_{p_1}^{(r)}(P_{n-1}^*)| + |\Ic_{p_1}^{(r-1)}(P_{n-1}^*)| + |\Ic_{p_1}^{(r-1)}(P_{n-2}^*)| + |\Ic_{p_1}^{(r-2)}(P_{n-2}^*)|\]
and
\[ |\Ic_{p_{n-1}}^{(r)}(P_n^*)| = |\Ic_{p_2}^{(r)}(P_{n-1}^*)| + |\Ic_{p_2}^{(r-1)}(P_{n-1}^*)| + |\Ic_{p_2}^{(r-1)}(P_{n-2}^*)| + |\Ic_{p_2}^{(r-2)}(P_{n-2}^*)|.\]
\end{remark}

By Theorem \ref{escape}, for any pendant graph $G^*$ there exists an $r$-star of maximum size whose centre is a pendant vertex. We include the proof of the special case $G=P_n^*$ here for completeness.

\begin{lemma}
\label{injections}
    For $1 \leq i \leq n$ and $1 \leq r \leq n$, 
        \[|\Ic_{x_i}^{(r)}(P_n^*)| \leq |\Ic_{p_i}^{(r)}(P_n^*)|.\]
\end{lemma}

\begin{proof}
    Fix $1\leq i\leq n$ and consider the family
    \[\Pc=\{S\backslash\{x_i\}\cup\{p_i\}\,:\,S\in\Ic_{x_i}^{(r)}(P_n^*)\}.\]
    For $S\in\Ic_{x_i}^{(r)}(P_n^*)$, $x_i\in S$ implies $p_i\not\in S$. Furthermore, $N(p_i)=\{x_i\},$ thus $S\backslash\{x_i\}\cup\{p_i\}$ is independent. Therefore $S\subseteq\Ic_{p_i}^{(r)}(P_n^*)$ which implies the result.
\end{proof}

Next, we characterize the behavior of $|\Ic^{(n)}(P_n^*)|$, the total number of independent $n$-sets of $P_n^*$. Prior to stating this result, we provide the definition of the Fibonacci numbers. These numbers are usually defined with the seed values $0$ and $1$. For the purposes of this paper, however, we define the Fibonacci numbers with a shift in indexing.
\begin{definition}
The Fibonacci sequence is defined by
\[F(n) = F(n-1) + F(n-2) \]
with $F(0)=1$ and $F(1)=2$.
\end{definition}
With this in hand, we present the following result.
\begin{lemma}
\label{fibo}
If $n \geq 0$, then $|\Ic^{(n)}(P_n^*)| = F(n).$
\end{lemma}
\begin{proof}
A base-pendant pair $x_i$ and $p_i$ cannot both be in an independent set. The pendant graph $P_n^*$ is comprised of exactly $n$ base-pendant pairs, thus any independent set of size $n$ must contain either $x_i$ or $p_i$ for all $i$. Furthermore, since $p_n\not\in S$ exactly when $x_n\in S$, then $x_{n-1}\not\in S$ since $x_nx_{n-1}\in E(P_n^*)$. Hence $p_n\not\in S$ if and only if $x_n,p_{n-1}\in S$, so we can create the following two subfamilies $\Ac$ and $\Bc$ such that $|\Ic^{(r)}(P_n^*)|=|\Ac|+|\Bc|$
\begin{align*}
    \Ac&=\{S\backslash\{p_n\}\,:S\in\Ic^{(r)}(P_n^*),\,p_n\in S\} \text{ and }\\
    \Bc&=\{S\backslash\{x_n,p_{n-1}\}\,:S\in\Ic^{(r)}(P_n^*),\,p_n\not\in S\}
\end{align*}
Note that $A\in\Ic^{(r-1)}(P_{n-1}^*)$ if and only if $A\cup\{p_n\}\in\Ic^{(r)}(P_{n}^*)$ since $N(p_n)\cap V(P_{n-1}^*)=\emptyset$. Similarly, $B\in \Ic^{(r-2)}(P_{n-2}^*)$ if and only if $B\cup\{p_{n-1},x_n\}\in\Ic^{(r)}(P_{n}^*)$ since $N(\{p_{n-1},x_n\})\cap V(P_{n-2}^*)=\emptyset$. Therefore $\Ac=\Ic^{(r-1)}(P_{n-1}^*)$ and $\Bc=\Ic^{(r-2)}(P_{n-2}^*)$ together with $|\Ic^{(r)}(P_n^*)|=|\Ac|+|\Bc|$ results in the recurrence
\[|\Ic^{(n)}(P_n^*)| = |\Ic^{(n-1)}(P_{n-1}^*)| + |\Ic^{(n-2)}(P_{n-2}^*)|.\]

Examining small values of $n$, we note $|\Ic^{(0)}(P_0^*)|=1$ since the empty set is, in all cases, the only independent set of cardinality 0. We also have $|\Ic^{(1)}(P_1^*)| = 2$ as $P_1^*$ contains two vertices and there must consequently be two singleton independent sets.
\end{proof}
\begin{lemma}
\label{countInd}
If $1 \leq k \leq n$, then
\begin{align*}
    |\Ic^{(n)}_{p_k}(P_n^*)| =  F(k-1)F(n-k).
\end{align*}
\end{lemma}

\begin{proof}
Fix $1\leq k\leq n$ and consider the following two subgraphs of $P_n^*:$
\begin{align*}
    P_{<k}^*&:=P_n^*[\{x_1,p_1,\ldots,x_{k-1},p_{k-1}\}]\\
    P_{>k}^*&:=P_n^*[\{x_{k+1},p_{k+1},\ldots,x_n,p_n\}].
\end{align*}
 Recall that an independent $n$-set of $P_n^*$ contains exactly one of $x_i$ or $p_i$ for all $i$. For $S\in\Ic^{(r)}_{p_k}(P_n^*)$, this implies $|S\cap P_{<k}^*|=k-1$ and $|S\cap P_{>k}^*|=n-k$.
 
Furthermore, the inclusion of $p_k$ in $S$ implies that $x_k\not\in S$. Thus since the subgraphs $P_{<k}^*$ and $P_{>k}^*$ are only adjacent to $x_k$ in the larger graph $P_n^*$, $S$ can be realized as $S=A\cup B\cup\{p_k\}$ where $A\in\Ic^{(k-1)}(P_{<k}^*)$ and $B\in\Ic^{(n-k)}(P_{>k}^*)$. Therefore
\begin{align}
\label{fiboeq}
    |\Ic^{(n)}_{p_k}(P_n^*)| = |\Ic^{(k-1)}(P_{<k}^*)||\Ic^{(n-k)}(P_{>k}^*)|. 
\end{align}
Finally, notice that $P_{<k}^*$ is isomorphic to $P^*_{k-1}$ and $P_{>k}^*$ is isomorphic to $P^*_{n-k}$. Hence the result follows from applying Lemma \ref{fibo} to Equation \ref{fiboeq}.
\[
    |\Ic^{(n)}_{p_k}(P_n^*)| = F(k-1)F(n-k).
\qedhere
\]
\end{proof}

Now, we identify which value(s) of $k$ maximize the product $F(k-1)F(n-k)$. 
This allows us to identify stars of maximum size when $r=n$.

\begin{lemma}\label{FiboMaximum}
Fix $n\geq 1$ and $1\leq k \leq n$. Define 
\begin{align*}
    f(k) \coloneqq F(k-1)F(n-k).
\end{align*}
The function $f(k)$ is maximized when $k=2$ or $k=n-1$. 
\end{lemma}
\begin{proof}
First, we leverage the symmetry of the product,
\begin{align*}
    f(k) &= F(k-1)F(n-k), \\
    &= F(n-(n-k+1))F((n-k+1)-1), \\
    &= f(n-k+1).
\end{align*}
Setting $k=2$, this implies $f(2)=f(n-1)$. Therefore, it suffices to show that $k=2$ maximizes $F(k-1)F(n-k)$ for $k \leq \lceil n/2 \rceil$. It is well-known that the Fibonacci numbers can be written in closed form as 
\begin{align*}
    F(n-2) &= \frac{\varphi^{n} - (-\varphi)^{-n}}{\sqrt{5}},
\end{align*}
where $\varphi = \frac{1 + \sqrt{5}}{2}$ is the golden ratio \cite{Binet}. The product can then be written as 
\begin{align*}
    F(k-1)F(n-k) &= \frac{\left(\varphi^{k+1} - (-\varphi)^{-k-1}\right)\left(\varphi^{n-k+2} - (-\varphi)^{-n+k-2}\right)}{5}
    \end{align*}
    which equates to
    \begin{align*}
   5F(k-1)F(n-k) &= \varphi^{n+3} +(-1)^{n-k+1} \varphi^{-n+2k-1} +(-1)^{k} \varphi^{n-2k+1} + (-1)^{n+1} \varphi^{-n-3}.
\end{align*}
Since the terms $\varphi^{n+3}$ and $ (-1)^{n+1} \varphi^{-n-3}$ are independent of $k$, in order to maximize the product of $F(k-1)$ and $F(n - k)$ it suffices to maximize
\begin{align*}
    \alpha(k) \coloneqq (-1)^{n-k+1} \varphi^{-n+2k-1} +(-1)^{k} \varphi^{n-2k+1}.
\end{align*}
If $n$ is even, then
\begin{align*}
    \alpha(k) = (-1)^{-k+1} \varphi^{-(n-2k+1)} + (-1)^{k}\varphi^{n-2k+1}.
\end{align*}
\begin{claim}
If $k$ is odd, $\alpha(k)<0$ and if $k$ is even then $\alpha(k)>0$.
\end{claim}
\begin{claimproof}
We further refine our assumptions based on the parity of $k$: 
\begin{align}
\label{alpha}
\alpha(k) = \begin{cases}
-\varphi^{-(n-2k+1)} + \varphi^{n-2k+1} \quad &k \text{ is even}, \\
\varphi^{-(n-2k+1)} - \varphi^{n-2k+1} \quad &k \text{ is odd.}
\end{cases}
\end{align}
Since $1\leq k\leq \lceil n/2\rceil$, we have $1\leq n-2k+1\leq n-1$. Therefore, 
\begin{align}
\label{fiboinv}
\varphi^{-(n-2k+1)} < 1
\end{align} 
and
\begin{align}
\label{fiboinv2}
    \varphi^{n-2k+1} > 1.
\end{align}
Applying Equations \eqref{fiboinv} and \eqref{fiboinv2} to Equation \eqref{alpha} for odd $k$ implies
\begin{align*}
\alpha(k) &< 1 - \varphi^{n-2k+1} < 0.
\end{align*}
Similarly, we get a lower bound for even $k$,
\begin{align*}
\alpha(k) &> -1 + \varphi^{n-2k+1} > 0. 
\end{align*}
So we conclude $\alpha(k)$ is positive for even $k$ and negative for odd $k$.
\end{claimproof}
In order to maximize $\alpha(k)$ it is sufficient to consider even $k$. The derivative of $\alpha(k)$ for even $k$ is
\begin{align*}
\frac{d \alpha}{dk} = -2\ln(\varphi) \varphi^{-(n-2k+1)} - 2\ln(\varphi) \varphi^{n-2k+1}.
\end{align*}
Noting that $\ln(\varphi) > 0$ yields
\begin{align*}
\frac{d \alpha}{dk} &< 0.
\end{align*}
Therefore, $\alpha(k)$ is a strictly decreasing function, hence it is maximized by the minimum value of $k$, which is $k=2$.

\ \\
Finally, we consider the case where $n$ is odd. As in the $n$ is even case, we again refine $\alpha(k)$ based on the parity of $k$: 
\begin{align*}
\alpha(k) = \begin{cases} 
\varphi^{-(n-2k+1)} + \varphi^{n-2k+1} \quad & k \text{ is even}, \\
-\varphi^{-(n-2k+1)} - \varphi^{n-2k+1} \quad & k \text{ is odd}.
\end{cases}
\end{align*}
It is clear that $\alpha(k)>0$ when $k$ is even, while $\alpha(k)<0$ when $k$ is odd. Hence, we consider only even $k$ and compute the derivative:
\begin{align*}
\frac{d \alpha}{dk} = 2\ln(\varphi) \varphi^{-(n-2k+1)} - 2\ln(\varphi) \varphi^{n-2k+1}.
\end{align*}
Applying Equations \eqref{fiboinv}, \eqref{fiboinv2} and that $\ln(\varphi)>0$, we obtain
\begin{align*}
\frac{d \alpha}{dk} &< 2\ln(\varphi) - 2\ln(\varphi) = 0.
\end{align*}
Since $\alpha(k)$ is a strictly decreasing function, it is maximized when $k$ is minimum. Thus $\alpha(k)$ and hence $f(k)$ is maximized when $k=2$ and, by symmetry, when $k=n-1$. 
\end{proof}

Applying Lemmas \ref{countInd} and \ref{FiboMaximum} gives us the following corollary.

\begin{cor}
\label{cor:nmax}
$|\Ic^{(n)}_{p_k}(P_n^*)|$ is maximized when $k=2$ or $k=n-1$.
\end{cor}

Next we are ready to prove our second main result that $\Ic_{p_2}^{(r)}(P_n^*)$ and $\Ic_{p_{n-1}}^{(r)}(P_n^*)$ are each of maximum size for all $n$ and $r\leq n$.

\begin{proof}[Proof of Theorem \ref{thetheorem}]
By Lemma \ref{injections} it suffices to show $|\Ic^{(r)}_{p_i}(P_n^*)|\leq|\Ic^{(r)}_{p_2}(P_n^*)|$ for all $i$. We proceed by induction on $r$ and begin by establishing two base cases.
When $r=1$, it is clear that for all $i$, $|\Ic_{p_i}^{(1)}(P_n^*)|= 1$. For $r=2$, all independent $r$-sets containing $p_i$ are of the form $\{p_i,p_j\}$ or $\{p_i,x_j\}$ for $i\neq j$.
Thus, $|\Ic^{(2)}_{p_i}(P_n^*)| = 2n-2$ for all $i$. Thus Equation \eqref{fiboeq} holds for $r=1$ and $r=2$.

Fix $i$ and $3\leq k\leq n$ and assume that for all $r<k$,
$|\Ic_{p_i}^{(r)}(P_n^*)|\leq |\Ic_{p_2}^{(r)}(P_n^*)|$. Further induct on $n\geq k$. By Corollary \ref{cor:nmax}, we have $|\Ic_{p_i}^{(k)}(P_k^*)| \leq |\Ic_{p_2}^{(k)}(P_k^*)|$. Now fix $l$ and suppose that for all $r$, $j$ such that $2\leq r< k \leq j <l$, $|\Ic_{p_i}^{(r)}(P_{j}^*)| \leq |\Ic_{p_2}^{(r)}(P_{j}^*)|$ holds.
Taking $r=k-1$ and $j=l-1$ we have $|\Ic_{p_i}^{(k-1)}(P_{l-1}^*)| \leq |\Ic_{p_2}^{(k-1)}(P_{l-1}^*)|$, hence
\begin{align*}|\Ic_{p_i}^{(k-1)}(P_{l-1}^*)| + |\Ic_{p_i}^{(k)}(P_{l-1}^*)| \\\leq |\Ic_{p_2}^{(k-1)}(P_{l-1}^*)| + |\Ic_{p_2}^{(k)}(P_{l-1}^*)|.\end{align*}
Similarly, $ |\Ic_{p_i}^{(k-1)}(P_{l-2}^*)| \leq |\Ic_{p_2}^{(k-1)}(P_{l-2}^*)|$ which gives
\[|\Ic_{p_i}^{(k-1)}(P_{l-2}^*)| + |\Ic_{p_i}^{(k-1)}(P_{l-1}^*)| + |\Ic_{p_i}^{(k)}(P_{l-1}^*)| \leq |\Ic_{p_2}^{(k-1)}(P_{l-2}^*)| + |\Ic_{p_2}^{(k-1)}(P_{l-1}^*)| + |\Ic_{p_2}^{(k)}(P_{l-1}^*)|.\]
Applying the inductive hypothesis once more, we have $|\Ic_{p_i}^{(k-2)}(P_{l-2}^*)| \leq |\Ic_{p_2}^{(k-2)}(P_{l-2}^*)|$, hence
\begin{align*}|\Ic_{p_i}^{(k-2)}(P_{l-2}^*)| + |\Ic_{p_i}^{(k-1)}(P_{l-2}^*)| + |\Ic_{p_i}^{(k-1)}(P_{l-1}^*)| + |\Ic_{p_i}^{(k)}(P_{l-1}^*)| \\\leq |\Ic_{p_2}^{(k-2)}(P_{l-2}^*)| + |\Ic_{p_2}^{(k-1)}(P_{l-2}^*)| + |\Ic_{p_2}^{(k-1)}(P_{l-1}^*)| + |\Ic_{p_2}^{(k)}(P_{l-1}^*)|.\end{align*}
Since $k \geq 3$ and $l > k$, by Equation \eqref{bigrecrelation} we have that
\[|\Ic_{p_i}^{(k)}(P_{l}^*)| \leq |\Ic_{p_2}^{(k)}(P_l^*)|.\]
\end{proof}

Now that we have identified the maximum $r$-stars of a pendant path graph, we look to investigate for which value(s) of $r$ is $P_n^*$ $r$-EKR. In doing so, we find that when $n\geq 4$ $P^*_n$ is not $n$-EKR.

\begin{lemma}
\label{notNEKR}
For all $n\geq 4$, $P^*_n$ is not $n$-EKR.
\end{lemma}
\begin{proof} 
Recall that an independent $n$-set $S \in \Ic^{(r)}(P_n^*)$ contains exactly one of either $p_i$ or $x_i$ for all $i$. 
Let $A \in \Ic^{(n)}(P_n^*)$ be the independent set that contains $x_i$ if $i$ is odd, and $p_i$ if $i$ is even. Then the complement of $A$, denoted $A^{C}$, is the independent set containing $x_i$ if $i$ is even, and $p_i$ if $i$ is odd. Any intersecting independent family of $n$-sets cannot contain both $A$ and $A^C$, so the largest such family has size at most $|\Ic^{(n)}(P_n^*)|-1$. We show that this bound is tight.

Define the family of independent $r$-sets $\Bc = \Ic^{(n)}(P_n^*) \setminus \{A^{C}\}$. We claim $\Bc$ is intersecting. 
Note that any independent set $S \in \Bc\backslash \{A\}$ must contain more pendant vertices than that of $A$. Thus, any two sets in $\Bc\backslash\{A\}$ contain at least $n+1$ pendant vertices between them. Since there are exactly $n$ pendant vertices, the two sets must intersect by the Pigeonhole Principle. Moreover, since $A^C$ is the only independent $r$-set that does not intersect $A$ and $A^C\not\in\Bc$ then $\Bc$ is intersecting. By Lemma \ref{fibo}, this implies that the largest set of intersecting independent $n$-sets has size $F(n)-1$.  

Finally, by Lemma \ref{countInd} we have $\Ic^{(n)}_{p_2}(P_n^*) = 2F(n-2)$. 
Since $F(n)-1 > 2F(n-2)$ when $n\geq 4$, we conclude that the pendant path graph $P^*_n$ is not $n$-EKR.
\end{proof}

\section{Future work}\label{sec:futurework}

It is important to note that although the pendant path graph is similar in many ways to the ladder graph, the ladder graph has been proven to be $n-$EKR in \cite{LadderGraph} while we see here that the pendant path graph is not. However we do still believe that Holroyd and Talbot's conjecture holds for pendant path graphs, which we conjecture here.

\begin{conjecture}
The pendant path graph $P^*_n$ is $r$-EKR for $n \geq 2r$. 
\end{conjecture}

Although the standard compression method seems to fail for the pendant path graph, another technique may be suitable such as the cycle method of \cite{cycle_method}. Furthermore, beyond the pendant path graph there is certainly much left to investigate in regards to the Erd\H{o}s-Ko-Rado property of pendant graphs.

\bibliographystyle{plain}
\bibliography{Bibliography.bib}
\end{document}